\documentclass[11pt,oneside]{amsart}
\usepackage{amsfonts, amssymb}
\usepackage{amsmath,amsthm,amscd,xypic}
\usepackage{color}
\usepackage{epic,eepic}
\usepackage{hyperref}
\usepackage{geometry}
\usepackage{graphicx}
\newtheorem{theorem}{Theorem}[section]

\newtheorem*{thm*}{Theorem}
\newtheorem{lem}[theorem]{Lemma}
\newtheorem{cor}[theorem]{Corollary}
\newtheorem*{cor*}{Corollary}

\newtheorem*{thmA}{Theorem A}

\theoremstyle{definition}

\newtheorem*{conj*}{Conjecture}

\numberwithin{equation}{section}

\newcommand{\vol}[1]{\text{vol}\left(#1\right)}

\newcommand{\diam}{\text{diam}}

\newcommand{\Rnkv}{\mathcal M(n,\kappa,v)}
\newcommand{\Lnkv}{\mathcal M_\infty(n,\kappa,v)}

\newcommand{\dsp}{\displaystyle}
\newcommand{\geod}[1]{[\,#1\,]}

\begin{document}


\title{Lipschitz-Volume rigidity on limit spaces with Ricci curvature bounded from below}

\author{Nan Li}
\address{Department of Mathematics, The Penn State University, University Park, PA 16802}
\email{nul12@psu.edu, lilinanan@gmail.com}
\urladdr{\href{https://sites.google.com/site/mathnanli/}
{https://sites.google.com/site/mathnanli/}}

\author{Feng Wang}
\address{Department of Mathematics, School of Mathematical Science of Peking University, China}
\email{fengwang232@gmail.com}





\maketitle

\begin{abstract}
  We prove a Lipschitz-Volume rigidity theorem for the non-collapsed Gromov-Hausdorff limits of manifolds with Ricci curvature bounded from below. This is a counterpart of the Lipschitz-Volume rigidity in Alexandrov geometry.
\end{abstract}

\section*{Introduction}

Let $\Rnkv$ be the collection of $n$-dimensional Riemannian manifolds $M$ with Ricci curvature bounded from below by $-(n-1)\kappa$ and $\vol{B_1(p)}\ge v$ for all $p\in M$. By Cheeger-Gromov Compactness Theorem, $\Rnkv$ is pre-compact in the pointed Gromov-Hausdorff topology. Let $\Lnkv$ be the closure of $\Rnkv$.

We let $``\text{vol}"$ denote the $n$-dimensional Hausdorff measure with an appropriate normalization. It has been proved in \cite{Col97} and \cite{CC97-I} that $\vol{\cdot}$ is a continuous function over balls with respect to the Hausdorff distance. T. Colding \cite{Col96-2} proved that an $n$-dimensional Riemannian manifold with Ricci curvature $\ge(n-1)$ is Hausdorff close to the unit sphere $\mathbb S^n$ if and only if its volume is close to the volume of $\mathbb S^n$. In this paper, we prove the following Lipschitz-Volume rigidity theorem.

\begin{thmA}[LV-rigidity]
  Let $X, Y\in \Lnkv$. Suppose that there is a 1-Lipschitz map $f\colon X\to Y$. If $\vol X=\vol {f(X)}$, then $f$ is an isometry with respect to the intrinsic metrics of $X$ and $f(X)$. In particular, if $f$ is also onto, then $Y$ is isometric to $X$.
\end{thmA}

Since $\Rnkv$ is compact in Gromov-Hausdorff topology, when $\vol X$ and $\vol Y$ are close to each other, an almost distance non-increasing map $f:X\to Y$ should be an almost isometry, that is, a Gromov-Hausdorff approximation. A map $f:X\to Y$ is called an $\epsilon$-Gromov-Hausdorff approximation if it is $\epsilon$-onto and $\epsilon$-isometry, that is,
\begin{align*}
  Y\subset B_\epsilon(f(X)) \text{\quad and \quad}
  \left||f(x)f(y)|_Y-|xy|_X\right|<\epsilon.
\end{align*}
It is not difficult to see that $d_{GH}(X,Y)\le 3\epsilon$ if there exists an $\epsilon$-Gromov-Hausdorff approximation $f:X\to Y$.


\begin{cor}
  For any $n,\kappa,v,D>0$ and $\epsilon>0$, there exists $\delta=\delta(n,\kappa,v,D,\epsilon)>0$ such that for any $X,Y\in \Lnkv$ that satisfies $\max\{\diam(X), \diam (Y)\}\le D$ and $|\vol X-\vol Y|<\delta$, a map $f:X\to Y$ is an $\epsilon$-Gromov-Hausdorff approximation if and only if it is $\delta$-distance non-increasing, that is, $|f(x)f(y)|_Y\le |xy|_X+\delta,$ for all $x,y\in X$.
\end{cor}

%

Theorem A fails for general length metric spaces, partially because a lemma of dimension control (Lemma \ref{dim.comp}) fails in these cases. For example, $Y$ can be the quotient space of $X$ with any lower dimensional subset identified as one point. See \cite{Li12} for more examples. A special case of Theorem A was proved by Bessi\`{e}res, Besson, Courtois, and Gallot in \cite{BBGG}. This was used to prove a stability theorem for certain manifolds with relatively minimum volume. Similar results and applications in Alexandrov geometry (c.f. \cite{BGP}) can be found in \cite{Li12}.

\bigskip

{\bf Conventions and notations}
\begin{itemize}
  \item $d_{GH}(X,Y)$: the Gromov-Hausdorff distance between $X$ and $Y$.
  \item $\dim_H(X)$: the Hausdorff dimension of $X$.
  \item $B_r^n(0)$: the $n$-dimensional Euclidean ball with radius $r$.
  \item $\mathcal R(X)$: the regular set in $X$, see \cite{CC97-I}.
  \item $\mathcal S(X)$: the singular set in $X$, see \cite{CC97-I}.
  \item $\geod{ab}_X$: a length minimizing geodesic connecting $a$ and $b$ in $X$.
  \item $|ab|_X$: the distance of $a$ and $b$ with respect to the intrinsic metric of $X$.
  \item $\psi(\epsilon)$: a function (could be different even in the same context) depending only on $n,\kappa,v,\epsilon$ that satisfies $\dsp\lim_{\epsilon\to 0}\psi(\epsilon)=0$.
\end{itemize}

\section{Lipschitz-Volume rigidity theorem}

Not losing generality, we assume that $f$ is onto and $\diam(Y)\le\diam(X)\le D$, since our proof only relies on the local structures. For simplicity, we only consider $X\in\mathcal M_\infty(n,-1,v)$. By the assumption, it's not hard to see that for any subset $A\subset Y$, $\vol A=\vol{f^{-1}(A)}$. One of the key step in our proof is to show that $f$ is injective (Lemma \ref{f.inj}). We first prove this for $f$ restricted to the regular part. Let $\mathcal R_{\epsilon,\delta}(X)=\{p\in X:\, d_{GH}(B_r(p),B_r^n(0))<\epsilon r \text{ for all $0<r<\delta$}\}$ and $\mathcal R_\epsilon(X)=\cup_\delta\mathcal {R}_{\epsilon,\delta}(X)$ be the $\epsilon$-regular set. By the volume continuity, we know that for any $x\in \mathcal R_{\epsilon,\delta}$ and $r<\delta$,
  \begin{align*}
    (1+\psi(\epsilon))\cdot\text{vol}(B_r(x))&=\vol{B_r^n(0)}
     =\vol{\mathbb S_1^{n-1}}\int_0^r t^{n-1}\,dt  \\
     &=2r\cdot\vol{B_r^{n-1}(0)}\int_{0}^{\frac\pi2}\sin^n(t)dt.
  \end{align*}

\begin{lem}\label{delta.img}
$f(\mathcal R_{\epsilon,\delta}(X))\subset \mathcal R_{\psi(\epsilon),\delta/10}(Y)$. Consequently, $f(\mathcal R_\epsilon(X))\subset \mathcal R_{\psi(\epsilon)}(Y)$ and $f(\mathcal R(X))\subset \mathcal R(Y)$.
\end{lem}

\begin{proof}
  Let $x\in \mathcal R_{\epsilon,\delta}(X)$ and $y=f(x)$. Apply the volume formula for $B_{\delta/10}(x)$. Because $f$ is volume preserving and $f^{-1}((B_{\delta/10}(y))\supseteq B_{\delta/10}(x)$, we have the following volume comparison:
  \begin{align*}
    \vol{B_{\delta/10}(y)}
    &= \vol{f^{-1}(B_{\delta/10}(y))}
    \\
    &\ge \vol{B_{\delta/10}(x)}
    =(1+\psi(\epsilon))\cdot\vol{B_{\delta/10}^n(0)}.
  \end{align*}
  By the almost maximum volume theorem \cite{Col96-1}, $y\in R_{\psi(\epsilon),\delta/10}(Y)$.
\end{proof}

Let $G_Y=\left\{y\in Y: f^{-1}(y) \text{\, has a cardinality of more than $1$}\right\}$ and $G_X=f^{-1}(G_Y)\subset X$.

\begin{lem}\label{delta.inj}
There is an $\epsilon=\epsilon(n,v)>0$ such that $\mathcal R_\epsilon(X)\cap G_X=\varnothing$. 

\end{lem}

\begin{proof}
  Argue by contradiction. Assume $x_1\in R_{\epsilon,\delta_0}(X)$, $x_2\in X$ and $f(x_1)=f(x_2)=y$. By Lemma \ref{delta.img}, $y\in R_{\psi(\epsilon),\delta_0/10}(Y)$. Let $0<\delta<\delta_0/10$ be small such that $B_\delta(x_1)\cap B_\delta(x_2)=\varnothing$. By the volume continuity, Bishop-Gromov Relative Volume Comparison holds on $X$ and $Y$. Thus we have
  \begin{align*}
    1&=\frac{\vol{f^{-1}(B_\delta(y))}}{\vol{B_\delta(y)}}
    \ge \frac{\vol{B_\delta(x_1)}+\vol{B_\delta(x_2)}}{\vol{B_\delta(y)}}
    \\
    &\ge \frac{\vol{B_\delta(x_1)} +v\cdot\frac{\int_0^\delta \sinh^{n-1}(t)\,dt}{\int_0^D \sinh^{n-1}(t)\,dt}} {\vol{B_\delta(y)}}
    \\
    &\ge \frac{(1+\psi(\epsilon))\cdot\vol{\mathbb S_1^{n-1}}\cdot \int_0^\delta t^{n-1}\,dt +v\cdot\frac{\int_0^\delta \sinh^{n-1}(t)\,dt}{\int_0^D \sinh^{n-1}(t)\,dt}} {(1+\psi(\epsilon))\cdot\vol{\mathbb S_1^{n-1}}\cdot \int_0^\delta t^{n-1}\,dt}.
  \end{align*}
  Let $\delta\to 0$, we get
  $$1\ge \frac{(1+\psi(\epsilon))\cdot\vol{\mathbb S_1^{n-1}}
    +\frac{v}{\int_0^D \sinh^{n-1}(t)\,dt}}
    {(1+\psi(\epsilon))\cdot\vol{\mathbb S_1^{n-1}}}.
  $$
  This is a contradiction for $\epsilon>0$ sufficiently small.
\end{proof}

In the next step, we prove that $f$ is almost isometry when restricted to the regular part. We need a volume formula for the union of two balls, witch follows by the volume continuity and direct computations in Euclidean space.

\begin{lem}\label{vol.2ball} For any $x_1, x_2\in \mathcal R_{\epsilon,\delta}$ with $|x_1x_2|\le 2r<\delta/5$,
        \begin{align*}
          &(1+\psi(\epsilon))\cdot
          \vol{B_r(x_1)\cup B_r(x_2)}
          \\
          &\qquad=\vol{B_r^n(0)} +2r\cdot\vol{B_r^{n-1}(0)}\int_{\theta}^{\frac\pi2}\sin^n(t)dt,
        \end{align*}
        where $\theta=\cos^{-1}\left(\frac{|x_1x_2|}{2r}\right)$.
\end{lem}

Now we can prove that $f|_{\mathcal R_{\epsilon,\delta}}$ is locally almost isometry and the proof is similar as in \cite{Li12}. We include it here for the convenience to the readers.

\begin{lem}\label{delta.lip} There are $\epsilon,\delta>0$ sufficiently small so that if $y_1, y_2\in f(\mathcal R_{\epsilon,\delta}(X))$ with $|y_1y_2|_Y<\delta/20$, then
$$|f^{-1}(y_1) f^{-1}(y_2)|_X<(1+\psi(\epsilon))\cdot|y_1y_2|_Y.$$
\end{lem}

\begin{proof}
  Let $|f^{-1}(y_1)f^{-1}(y_2)|_X=\lambda\cdot|y_1y_2|_Y$.
  Consider the metric balls $B_r(y_1)$ and $B_r(y_2)$. Take $r=\frac12\lambda\cdot|y_1y_2|_Y$ and assume that $r<\delta/10$. By
  the volume formula in Lemma \ref{vol.2ball},
\begin{align*}
  &(1+\psi(\epsilon))\cdot\vol{B_r(y_1)\cup B_r(y_2)}
  \\
  &\qquad=\vol{B_r^n(0)}
  +2r\cdot\vol{B_r^{n-1}(0)}
  \int_{\theta}^{\pi/2}\sin^n (t)\,dt
  \\
  &\qquad=2r\cdot\vol{B_r^{n-1}(0)}
  \int_{0}^{\pi/2}\sin^n (t)\,dt
  +2r\cdot\vol{B_r^{n-1}(0)}
  \int_{\theta}^{\pi/2}\sin^n (t)\,dt,
\end{align*}
where $\theta=\cos^{-1}\left(\frac{|y_1y_2|_Y}{2r}\right) =\cos^{-1}\left(1/\lambda\right)$.
Note that $B_r(f^{-1}(y_1))\cap B_r(f^{-1}(y_2))=\varnothing$. We have
\begin{align*}
  &(1+\psi(\epsilon))\cdot\vol{B_r(f^{-1}(y_1))\cup B_r(f^{-1}(y_2))}
  \\
  &\qquad=2\vol{B_r^n(0)}
  =4r\cdot\vol{B_r^{n-1}(0)}
  \int_{0}^{\pi/2}\sin^n (t)\,dt.
\end{align*}
Because $f$ is 1-Lipschitz, we have $f^{-1}(B_r(y_1)
\cup B_r(y_2))\supseteq B_r
(f^{-1}(y_1))\cup B_r(f^{-1}(y_2))$. Together with that $f$ is volume preserving, we get
\begin{align}
  1&=\frac{\vol{f^{-1}(B_r(y_1)\cup B_r(y_2))}}
  {\vol{B_r(y_1)\cup B_r(y_2)}}
  \geq\frac{\vol{B_r(f^{-1}(y_1))\cup B_r(f^{-1}(y_2))}}
  {\vol{B_r(y_1)\cup B_r(y_2)}}
  \notag\\
  &=(1-\psi(\epsilon))\frac{2\int_{0}^{\pi/2}\sin^n (t)\,dt}
    {\int_{0}^{\pi/2}\sin^n (t)\,dt
    +\int_{\theta}^{\pi/2}\sin^n (t)\,dt}.
  \label{delta.lip.e1}
\end{align}
We claim that $\lambda\le 2$. If this is not true, we repeat the above calculation with $|f^{-1}(y_1)f^{-1}(y_2)|_X>2|y_1y_2|_Y$ and $r=|y_1y_2|_Y$. In this case $\theta=\frac\pi3$, which yields a contraction when $\epsilon$ is small. Once the claim is proved, the assumption $r<\delta/10$ automatically holds and then inequality (\ref{delta.lip.e1}) follows for all $|y_1y_2|_Y<\delta/20$.  This implies that $0<\theta<\psi(\epsilon)$ and thus $\lambda=\frac1{\cos\theta}<1+\psi(\epsilon)$.
\end{proof}

To prove that $f$ almost preserves the length of path for any curve $\gamma\subset \mathcal R_\epsilon(X)$, we need the existence of $\delta_0>0$ so that $\gamma\subset \mathcal R_{\epsilon,\delta_0}(X)$. Note that $\mathcal R_{\epsilon,\delta}(X)$ may not be open, but by the continuity of volume and the rigidity of almost maximal volume, we get that for any $\epsilon,\delta>0$ small, there is $\epsilon_1=\psi(\epsilon_1)<\epsilon$ so that $\mathcal R_{\epsilon,\delta}(X)\subset \overset{\circ}{\mathcal R}_{\epsilon_1,\delta/2}(X)$. Thus
\begin{equation*}
  \mathcal R_{\epsilon}(X)
  =\cup_\delta \mathcal R_{\epsilon,\delta}(X)
  \subset \cup_\delta\overset{\circ}{\mathcal R}_{\epsilon_1,\delta/2}(X).
\end{equation*}
If a compact set $A\subset \mathcal R_\epsilon(X)$, then there is $\delta_0>0$ such that $A\subset\overset{\circ}{\mathcal R}_{\psi(\epsilon),\delta_0}(X)$. The following lemma is a direct consequence of Lemmas \ref{delta.img} -- \ref{delta.lip}.

\begin{lem}[Almost Isometry over $\mathcal R_{\epsilon}(X)$]\label{delta.almost.iso} There is $\epsilon>0$ small so that the following holds.
  \begin{enumerate}
    \item If $\geod{pq}_Y\subset f(\mathcal R_\epsilon(X))$, then $\gamma=f^{-1}(\geod{ab}_Y)$, parameterized by arc length, is a Lipschitz curve with $$L(\gamma|_{[t_1,t_2]})< (1+\psi(\epsilon))\cdot|\gamma(t_1)\gamma(t_2)|_{\mathcal R_\epsilon(X)}.$$
    \item $f|_{\mathcal R_\epsilon(X)}$ is $(1+\psi(\epsilon))$-Lipschitz. In particular, if geodesic $\geod{f(a)f(b)}_Y\subset f(\mathcal R_\epsilon(X))$, then
        \begin{equation*}
          1\le\frac{|ab|_X}{|f(a)f(b)|_Y}< 1+\psi(\epsilon).
        \end{equation*}
    \item $f(\mathcal R_\epsilon(X))\subseteq Y^{\psi(\epsilon)}$ is open and dense in $Y$.
  \end{enumerate}
\end{lem}

Now we can prove that $f$ is injective with the following Dimension comparison lemma (compare to \cite{Li12} for the Alexandrov case), which is a direct consequence of volume convergence and Lemma 3.1 in \cite{CC00-II}.

\begin{lem}[Dimension comparison]\label{dim.comp}
  Let $\Omega_0\subseteq X\in\Lnkv$ be a subset with $\vol{\Omega_0}>0$ and $p\in X$ be a fixed point. For each point $x\in \Omega_0$, select one point $\bar x$ on a geodesic $\geod{px}_X$. Let $\Omega$ be the collection of the $\bar x$s for all $x\in\Omega_0$. If $d(p,\Omega)>0$, then
  $$\dim_H(\Omega)\ge n-1.$$
\end{lem}

\begin{lem}\label{f.inj}
  $f\colon X\to Y$ is injective.
\end{lem}

\begin{proof}
  Assume $G_X\neq\varnothing$. Let $p,q\in G_X$ such that $f(p)=f(q)=a\in G_Y$. We will show that there exists $\epsilon>0$ such that $\dim_H(f(X\setminus\mathcal R_{\epsilon}(X)))\ge n-1$.
  Then $$\dim_H(\mathcal S(X))\ge \dim_H(X\setminus\mathcal R_{\epsilon}(X))
  \ge \dim_H(f(X\setminus\mathcal R_{\epsilon}(X)))\ge n-1,$$
  This contradicts to the fact that $\dim_H(\mathcal S(X))\le n-2$ proved in \cite{CC97-I}.

  By Lemma \ref{delta.inj}, let $\epsilon>0$ be small such that $\mathcal R_{\epsilon}\cap G_X=\varnothing$. Then $f(X\setminus\mathcal R_{\epsilon}(X))=Y\setminus f(\mathcal R_{\epsilon}(X))$ and $f(\mathcal R_{\epsilon}(X))\cap G_Y=\varnothing$. We shall show that
  \begin{align}
  \dim_H\left(Y\setminus f(\mathcal R_{\epsilon}(X))\right)\ge n-1.
  \label{f.inj.e1}
  \end{align}

  For any $\eta>0$ small, there is $q_1\in\mathcal R(X)$ with $|qq_1|_X<\eta$. Let $\Omega_0=f\left(B_{\eta}(p)\cap\mathcal R(X)\right)$. Because $f$ is volume preserving, we have \begin{align}\vol{\Omega_0}>0.\label{f.inj.e1.1}\end{align}

  Let $a_1=f(q_1)$. We claim that for any $y\in\Omega_0$, $\geod{ya_1}_Y\setminus f(\mathcal R_\epsilon(X))\neq\varnothing$. If not so, then $\geod{ya_1}_Y\subset f(\mathcal R_\epsilon(X))$. Let $x=f^{-1}(y)\in B_{\eta}(p)$. By Lemma \ref{delta.almost.iso},
  \begin{align}
    |ya_1|_Y&=(1-\psi(\epsilon))|xq_1|_X\notag
    \\
    &\ge (1-\psi(\epsilon))(|pq|_X-|px|_X-|qq_1|_X)
    \ge (1-\psi(\epsilon))(|pq|_X-2\eta).
    \label{f.inj.e2}
  \end{align}
  On the other hand,
  \begin{align} |ya_1|_Y\le |ya|_Y+|aa_1|_Y\le |xp|_X+|qq_1|_X\le 2\eta. \label{f.inj.e3}\end{align}
  This leads to a contradiction to (\ref{f.inj.e2}) for $\epsilon, \eta>0$ small.

  Now for each $y\in\Omega_0$, take $\bar y\in\geod{ya_1}\setminus f(\mathcal R_\epsilon(X))$. Let $\Omega$ be the collection of $\bar y$ for all $y\in\Omega_0$. Clearly $\Omega\subseteq Y\setminus f(\mathcal R_{\epsilon}(X))$. Note that $a_1=f(q_1)\in\mathcal R(Y)$ due to Lemma \ref{delta.img}. By Theorem A.1.5 in \cite{CC97-I}, $d(a_1,\Omega)\ge c(\epsilon)>0$. Together with (\ref{f.inj.e1.1}) and Lemma \ref{dim.comp}, we get that
  $$\dim_H\left(Y\setminus f(\mathcal R_{\epsilon}(X))\right)\ge\dim_H(\Omega)\ge n-1.$$
\end{proof}

\begin{proof}[{\bf Proof of Theorem A}]
  We first show that for any $a\in f(\mathcal R_\epsilon(X))$, $y\in Y$ and any $r>0$, there is $y'\in B_r(y)$ such that $\geod{ay'}_Y\subset f(\mathcal R_{2\epsilon}(X))$. If this is not true, then $\geod{ay'}_Y\setminus f(\mathcal R_{2\epsilon}(X))\neq\varnothing$ for any $y'\in B_r(y)$. Let $\bar y\in \geod{ay'}_Y\setminus f(\mathcal R_{2\epsilon}(X))$ and $\Omega$ be the collection of these $\bar ys$ for all $y'\in B_r(y)$. Clealy, $\Omega\subseteq f(\mathcal S(X))$ and $\dim_H(\Omega)\le n-2$. By the volume convergence of small Euclidean balls, we see that there is a constant $c=c(\epsilon)$ such that $|a\bar y|\ge c(\epsilon)$. By Lemma \ref{dim.comp}, we get $\dim_H(\Omega)\ge n-1$, a contradiction.

  Let $\gamma_X:[0,1]\to X$ be a Lipschitz curve and $\gamma_Y=f(\gamma_X)$. It's sufficient to show that $L(\gamma_Y)\ge L(\gamma_X)$. Let $\{y_i\}_{i=0}^N$ be a partition of $\gamma_Y$ with $\max\{|y_1y_{i+1}|\}\to 0$ as $N\to\infty$. For any $\epsilon>0$, let $\epsilon_1=\psi(\epsilon)>0$ be selected as in Lemma \ref{delta.almost.iso}. Take $y_0'\in B_{y_0}(\epsilon/N)\cap f(\mathcal R_{\epsilon_1/2^N}(X))\neq\varnothing$. Select $y_i'\in B_{y_i}(\epsilon/N)$ recursively such that $\geod{y_i'y_{i+1}'}_Y\subset f(\mathcal R_{\epsilon_1/2^{N-(i+1)}}(X))$ for $i=1,2,\dots,N-1$.

  Let $x_i=f^{-1}(y_i')$, $i=0,1,\dots,N$. Applying Lemma \ref{delta.almost.iso} to $\geod{y_i'y_{i+1}'}_Y\subset f(\mathcal R_{\epsilon_1/2^{N-(i+1)}}(X))\subseteq f(\mathcal R_{\epsilon_1}(X))$ we get
  \begin{align*}
    L(\gamma_Y)
    &\ge \sum_{i=0}^{N-1}|y_iy_{i+1}|_Y
    \ge\sum_{i=0}^{N-1}\left(|y_i'y_{i+1}'|_Y-\frac{2\epsilon}{N}\right)
    \ge \sum_{i=0}^{N-1}|y_i'y_{i+1}'|_Y-2\epsilon
    \\
    &\ge (1-\psi(\epsilon))\sum_{i=0}^{N-1}|x_ix_{i+1}|_X-2\epsilon.
  \end{align*}
  Let $N\to\infty$ with that $\cup_i\geod{y_iy_{i+1}}_Y$ converges to $\gamma_Y$, because $f$ is injective, $\cup_i\geod{x_ix_{i+1}}_X$ converges to $\gamma_X$. Thus taking $N\to\infty$ and $\epsilon\to 0$, we get
  $$L(\gamma_Y)\ge
    \liminf_{\substack{N\to\infty\\\epsilon\to 0}}
    \sum_{i=0}^{N-1}|x_ix_{i+1}|_X
    \ge L(\gamma_X).
  $$
\end{proof}


%

\vskip 30mm

\bibliographystyle{amsalpha}


\end{document}